\documentclass[11pt]{amsart}
\usepackage{amssymb, latexsym}
\theoremstyle{plain}
\newtheorem{theorem}{Theorem}

\newtheorem {lemma}{Lemma}
\newtheorem{proposition}{Proposition}

\theoremstyle{remark}

\newtheorem*{Remark 1}{Remark 1}
\newtheorem*{Remark 2}{Remark 2}
\newtheorem*{Remark 3}{Remark 3}
\newtheorem*{Remark 4}{Remark 4}

\numberwithin{equation}{section}

\begin{document}

\title[Brownian Motion with Random Jumps]%
 { Asymptotic Behavior of the Principal Eigenvalue for a Class of Non-Local Elliptic Operators   Related to Brownian Motion with Spatially Dependent Random Jumps}

\author{Nitay Arcusin and Ross G. Pinsky}
\address{Department of Mathematics\\
Technion---Israel Institute of Technology\\
Haifa, 32000\\ Israel} \email{nitay@techunix.technion.ac.il; pinsky@math.technion.ac.il}
\urladdr{http://www.math.technion.ac.il/~pinsky/}
\thanks{The  research of the second author  was supported by
 THE ISRAEL SCIENCE FOUNDATION (grant No.
449/07)}

\date{}

\begin{abstract}
Let $D\subset R^d$ be a bounded domain and let $\mathcal P(D)$
denote the space of probability measures  on $D$.  Consider a
Brownian motion in $D$ which is killed at the boundary and which,
while alive, jumps instantaneously according to a spatially dependent  exponential clock
with intensity $\gamma V$ to a new point,  according
to a distribution $\mu\in\mathcal P(D)$. From its new position after the jump, the process
repeats the above behavior independently of what has transpired
previously. The generator of this process is an extension of the
operator $-L_{\gamma,\mu}$, defined by
\begin{equation*}
L_{\gamma,\mu}u\equiv -\frac12\Delta u+\gamma V C_\mu(u),
\end{equation*}
with the Dirichlet boundary condition, where $C_\mu$ is the
  ``$\mu$-centering'' operator defined by
\begin{equation*}
C_\mu(u)=u-\int_Du~d\mu.
\end{equation*}
 The
 principal eigenvalue, $\lambda_0(\gamma,\mu)$, of $L_{\gamma,\mu}$ governs the exponential rate of decay of the probability of not
 exiting $D$ for large time. We study the asymptotic behavior of $\lambda_0(\gamma,\mu)$
as $\gamma\to\infty$. In particular, if $\mu$ possesses a density in a neighborhood of the boundary, which we call $\mu$, then
\begin{equation*}
\lim_{\gamma\to\infty}\gamma^{-\frac12}\lambda_0(\gamma,\mu)=\frac{\int_{\partial D}\frac\mu{\sqrt {V}}d\sigma}{\sqrt2\int_D\frac1{ V}d\mu}.
\end{equation*}
If $\mu$ and all its derivatives up to order $k-1$ vanish on the boundary, but the $k$-th derivative does not vanish identically on the boundary,
then $\lambda_0(\gamma,\mu)$ behaves asymptotically like $c_k\gamma^{\frac{1-k}2}$, for an explicit constant $c_k$.
\end{abstract}

\maketitle
\section{Introduction and Statement of Results}\label{S:intro}

Let $D\subset R^d$ be a bounded domain with $C^{2,\alpha}$-boundary ($\alpha\in (0,1]$) and let $\mathcal P(D)$ denote the space of probability measures on $D$.
Fix a measure $\mu\in\mathcal P(D)$, and
 consider a Markov process $X(t)$ in $D$ which performs Brownian motion and is killed at the
boundary, and which while alive, jumps instantaneously  according to a spatially dependent  exponential clock
with intensity $\gamma V$ to a new point,  according to the distribution $\mu$.
That is, the probability that the process $X(\cdot)$ has not jumped by time $t$, is given by $\exp(-\int_0^t\gamma V(X(s))ds)$.
 From its new position after the jump,
the process repeats the above behavior independently of what has
transpired previously.
 Let $\tau_D$ denote the  lifetime of the process. 
We assume that $V\in C^\alpha(\bar D)$ and that $V>0$ in $\bar D$, and we  normalize it by
$$
\int_DV(x)dx=1.
$$
 We will think of $V$ as being fixed and of $\gamma>0$ and $\mu$ as parameters that may be varied.
 Denote
probabilities and expectations for the process
starting from $x\in D$ by $P_x^{\gamma,\mu}$ and
$E_x^{\gamma,\mu}$.

Define the contraction semigroup
\begin{equation*}
T_t^{\gamma,\mu}f(x)=E_x^{\gamma,\mu}(f(X(t)); \tau_D>t),\  f\in C_0(\bar D),
\end{equation*}
where $C_0(\bar D)$ is the space of continuous functions on $\bar D$ vanishing on $\partial D$.
The infinitesimal generator of this semigroup is an extension of
the operator $-L_{\gamma,\mu}$, defined on
$C^2(\bar D)\cap\{u:u,L_{\gamma,\mu}u\in C_0(\bar D)\}$
by
\begin{equation*}
L_{\gamma,\mu}u\equiv -\frac12\Delta u+\gamma VC_\mu(u),
\end{equation*}
with the Dirichlet boundary condition, where $C_\mu$ is the
   ``$\mu$-centering'' operator defined by
\begin{equation*}
C_\mu(u)=u-\int_Du~d\mu.
\end{equation*}
The  operator $T_t^{\gamma,\mu}$ is  compact. These facts were proven in \cite{P09} in the case of constant $V$,
and can be proved similarly for variable $V$ as defined above. Since  $T_t^{\gamma,\mu}$ is  compact, the resolvent operator
for $T_t^{\gamma,\mu}$ is also  compact, and consequently  the spectrum $\sigma(L_{\gamma,\mu})$ of $L_{\gamma,\mu}$
consists exclusively of eigenvalues.
By the Krein-Rutman theorem, one deduces that $L_{\gamma,\mu}$ possesses a principal eigenvalue,
$\lambda_0(\gamma,\mu)$; that is, $\lambda_0(\gamma,\mu)$ is real and simple  and satisfies
$\lambda_0(\gamma,\mu)=\inf\{\text{Re}(\lambda):\lambda\in\sigma(L_{\gamma,\mu})\}$ \cite{P95}.
It is known that $\lambda\in\sigma(L_{\gamma,\mu})$
if and only if $\exp(-\lambda t)\in\sigma(T_t^{\gamma,\mu})$ \cite{Pazy83}.
Thus, since $||T^{\gamma,\mu}_t||<1$, it follows that $\lambda_0(\gamma,\mu)>0$.
We have
$$
\sup_{f\in C_0(\bar D),||f||\le1}||T_t^{\gamma,\mu}f||=\sup_{x\in D}P_x^{\gamma,\mu}(\tau_D>t);
$$
 thus, a standard result \cite{RS72} allows us to conclude that
$$
\lim_{t\to\infty}\frac1t\log \sup_{x\in D}P_x^{\gamma,\mu}(\tau_D>t)=-\lambda_0(\gamma,\mu).
$$
It is well known that this is equivalent to
\begin{equation}\label{largetime}
\lim_{t\to\infty}\frac1t\log P_x^{\gamma,\mu}(\tau_D>t)=-\lambda_0(\gamma,\mu),\ \ x\in D.
\end{equation}

The  Brownian motion with random jumps analyzed here is a paradigm for a phenomenon that occurs in various settings and which
is best illustrated perhaps in terms of computer-games or the game ``chutes and ladders.'' The object of the game is to reach the boundary of $D$
in as little time as possible (or alternatively, to avoid reaching the boundary for as much time as possible).
The game is played in rounds; however, time is always accumulating.
Various obstacles (modelled by the spatially dependent exponential clock with intensity $\gamma$) lead to the end of a round, and each new round
begins afresh from a new position which may be deterministic or random
(modelled by the measure $\mu$).
Then $\lambda_0(\gamma,\mu)$ is a measure of  the probability of long-term failure (or success, depending on the rules).
As $\gamma$ increases, the obstacles become more dense.

In \cite{P09}, the behavior of $\lambda_0(\gamma,\mu)$ was analyzed for the regimes $\gamma\ll1$ and $\gamma\gg1$
in the case of constant $V$. In this paper we consider the regime $\gamma\gg1$.
Note that probabilistic intuition suggests
the general direction of the result. Since $\gamma\gg1$,   the
Brownian motion  doesn't get very far before it jumps and gets
redistributed according to $\mu$. In particular then,  if
supp$(\mu)\subset D$,  it will be very difficult for the Brownian
motion to exit $D$, and in light of \eqref{largetime} one expects
that $\lim_{\gamma\to\infty}\lambda_0(\gamma,\mu)=0$. More
generally, one expects that the leading order asymptotic behavior for
large $\gamma$ will depend only on the behavior of $\mu$
arbitrarily close to the boundary.
For the case of constant $V$, in \cite{P09} it was shown that if $\mu$ is compactly supported in $D$, then there are constants $c_1,c_2$ such that
$\exp(-c_2\gamma^\frac12)\le\lambda_0(\gamma,\mu)\le \exp(-c_1\gamma^\frac12)$, for large $\gamma$.
Under the assumption that the measure $\mu$ possesses an appropriately smooth  density in a neighborhood of the boundary, which we will also call $\mu$,
it was proven in \cite{P09} that
\begin{equation}\label{P1}
\lim_{\gamma\to\infty}\gamma^{-\frac12}\lambda_0(\gamma,\mu)=\frac
1{\sqrt2|D|^\frac12}\int_{\partial D} \mu d\sigma.
\end{equation}
Assuming appropriate smoothness of the above density $\mu$, it was also proven there that if $\mu\equiv0$ on $\partial D$, then
\begin{equation}\label{P2}
\lim_{\gamma\to\infty}\lambda_0(\gamma,\mu)=\frac12\int_{\partial
D}(\nabla \mu\cdot n) d\sigma,
\end{equation}
while if $\mu,\nabla u\equiv0$ on $\partial D$, then
\begin{equation}\label{P3}
\lim_{\gamma\to\infty}\gamma^\frac12\lambda_0(\gamma,\mu)=\frac{|D|^\frac12}{2\sqrt2}\int_{\partial D}\Delta \mu d\sigma.
\end{equation}

The above results  show that in the case of constant $V$,
$\lambda_0(\gamma,\mu)$ grows on the order $\gamma^\frac12$ if the density $\mu$ of the jump measure does not vanish identically on $\partial D$,
while  for $k=1$ or 2, if all the derivatives of $\mu$ up to order $k-1$ vanish identically on the boundary, and at least one
of the  derivatives of order
$k$ does not vanish identically on $\partial D$, then $\lambda_0(\gamma,\mu)$
behaves asymptotically on the order $\gamma^\frac{1-k}2$. It is natural to expect that such behavior would continue for all positive integers $k$.

The case of variable $V$ is not at all a straight forward generalization of the constant case. To see why, consider first of all
what occurs if $V$ is allowed to be identically  0 in some sub-domain $A\subset D$. Then as long as the process remains in $A$, it never jumps;
consequently, starting at $x\in A$,  the probability of not exiting $D$ by time $t$ is greater than the probability of
a standard Brownian motion  not exiting $A$ by time $t$. In light of
\eqref{largetime}, this means that $\lambda_0(\gamma,\mu)\le\lambda_0^A$, where $\lambda_0^A$ is the principal eigenvalue for
$-\frac12\Delta$ in $A$ with the Dirichlet boundary condition. In particular, $\lambda_0(\gamma,\mu)$ is bounded and the behavior in
\eqref{P1} cannot occur.
 Now consider the case that $V$ is positive in $D$ but  decreases to 0
at $\partial D$. If this occurs at an appropriate rate, it should increase the tendency of the process
to leave the region, and thus raise the value of $\lambda_0(\gamma,\mu)$.
Indeed, in order for the process to exit the region, when the process is very near the boundary it needs to refrain from jumping.
Thus, in the case of variable $V$, the dependence on $V$ of the corresponding constant on the right hand side of \eqref{P1} should be consistent
with the above discussion.

In this paper, for variable, strictly positive $V$, we prove the analog of \eqref{P1} and the analog of a generalization of \eqref{P2}, \eqref{P3}
for the case that for some positive integer $k$, all the derivatives
of $\mu$ up to order $k-1$ vanish identically on $\partial D$.

\begin{theorem}\label{th}
Let $D\subset R^d$, $d\ge1$, be a bounded domain with a $C^{2,\alpha}$-boundary ($\alpha\in(0,1]$) and
let $\mu\in\mathcal P(D)$. Assume that $V>0$ on $\bar D$.
Let $\sigma$ denote Lebesgue measure on $\partial D$.
Let $D^\epsilon=\{x\in D: \text{dist}(x,\partial D)<\epsilon\}$.

 {\it\noindent i.}  Assume that for some
$\epsilon>0$, the restriction of $\mu$ to $D^\epsilon$ possesses a
density which belongs to $C^1(\bar D^\epsilon)$: $\mu(dx)|_{D^\epsilon}\equiv\mu(x)dx$.
Assume also that $V\in C^{2,\alpha}(\bar D)$. Then
\begin{equation}\label{1}
\lim_{\gamma\to\infty}\gamma^{-\frac12}\lambda_0(\gamma,\mu)=\frac{\int_{\partial D}\frac\mu{\sqrt {V}}d\sigma}{\sqrt2\int_D\frac1{ V}d\mu}.
\end{equation}

{\it\noindent ii.} Let $k\ge1$. Assume that
for some
$\epsilon>0$, the restriction of $\mu$ to $D^\epsilon$ possesses a
density which belongs to $C^{k+1}(\bar D^\epsilon)$: $\mu(dx)|_{D^\epsilon}\equiv\mu(x)dx$.
 Assume  also that $V\in C^{k+1}(\bar D)$ if $k$ is odd and that $V\in C^{k+1,\alpha}(\bar D)$
if $k$ is even.
Assume that
$$
\frac{d^\beta \mu}{d x^\beta}\equiv0 \ \text{on}\ \partial D,\ \text{ for all}\ |\beta|\le k-1.
$$
Let $n$ denote the inward unit normal to $D$ at $\partial D$.
\newline
If $k$ is odd, then
\begin{equation}\label{odd}
\lim_{\gamma\to\infty}\gamma^{\frac{k-1}2}\lambda_0(\gamma,\mu)=\frac{\int_{\partial D}
V^{-\frac{k+1}2}\nabla(\Delta^{\frac{k-1}2}\mu)\cdot nd\sigma}{2^{\frac{k+1}2}\int_D\frac1{ V}d\mu}.
\end{equation}
If $k$ is even, then
\begin{equation}\label{even}
\lim_{\gamma\to\infty}\gamma^{\frac{k-1}2}\lambda_0(\gamma,\mu)=\frac{\int_{\partial D}
V^{-\frac{k+1}2}\Delta^{\frac k2}\mu d\sigma}{2^{\frac{ k+1}2}\int_D\frac1{ V}d\mu}.
\end{equation}
\end{theorem}
\noindent \bf Remark.\rm\
As  $V$ decreases to 0 on some sub-domain $A\subset\subset D$
(and increases elsewhere in order to maintain the normalization $\int_DVdx=1$), assuming that supp$(\mu)\cap A\neq\emptyset$,  the constant on the right hand side
of \eqref{1} converges to 0, which is consistent with the discussion in the penultimate paragraph before Theorem \ref{th}.
(If on the other hand $A\cap$ supp$(\mu)=\emptyset$, then as $V$ decreases to 0 on $A$,  the constant
on the right hand side of \eqref{1} remains bounded away from 0. This is not inconsistent with the above-mentioned discussion; it shows
that the asymptotic behavior as $\gamma\to\infty$ is not uniform over $V$.)

Assuming that $\mu\not\equiv 0$ on $\partial D$,  if $V$ is of the form $\epsilon+(1-\epsilon|D|)\hat V$, where $\hat V$ is a smooth function which is strictly positive in $D$ and vanishes on $\partial D$,
then as $\epsilon\to0$, the right hand side of \eqref{1} converges to $\infty$. This is consistent with
the discussion in the penultimate paragraph before Theorem \ref{1}. It also suggests that for a smooth $V$
which is strictly positive in $D$ and vanishes on $\partial D$, $\lambda_0(\gamma,\mu)$ will grow on a larger order than $\gamma^\frac12$.
However, we cannot prove this, and it seems  conceivable to us that in fact the order of growth is smaller than $\gamma^\frac12$---see section 4.

We will also prove the following result in the case that $\mu$ is compactly supported.

\begin{proposition}\label{decay}
Let $\mu$ be compactly supported in $D$. Then there exist constants $c_1,c_2$ such that
\begin{equation}
\exp(-c_2\gamma^\frac12)\le\lambda_0(\gamma,\mu)\le \exp(-c_1\gamma^\frac12), \ \text{for}\ \gamma>1.
\end{equation}
\end{proposition}
In section 2 we present some preliminary results needed for the proof of Theorem \ref{th}, and
we conclude that section with the proof of Proposition \ref{decay}. Theorem \ref{th} is proved in section 3.
In section 4 we discuss an open problem concerning the behavior of $\lambda_0(\gamma,\mu)$ in the case
that $V$ is positive in $D$ but vanishes on the boundary.
\section{Preliminary Results and Proof of Proposition \ref{decay}}
In this section we prove a number of preliminary results, culminating in the proof of Proposition
\ref{decay}. Let $P_x$ and $E_x$ denote respectively probabilities and  expectations
for Brownian motion starting from $x$.

Let  $u_{\lambda,\gamma}$ and $v_{\lambda,\gamma}$ denote the  solutions to the equations
\begin{equation}\label{ue}
\begin{aligned}
\begin{cases}
&\frac12\Delta u_{\lambda,\gamma}+(\lambda-\gamma V(x))u_{\lambda,\gamma}=0\ \ \text{in}\ D;\\
&u_{\lambda,\gamma}=1\ \text{on}\ \partial D;
\end{cases}
\end{aligned}
\end{equation}

\begin{equation}\label{ve}
\begin{aligned}
\begin{cases}
&\frac12\Delta v_{\lambda,\gamma}+(\lambda-\gamma V(x))v_{\lambda,\gamma}=-1\ \ \text{in}\ D;\\
&v_{\lambda,\gamma}=0\ \text{on}\ \partial D.
\end{cases}
\end{aligned}
\end{equation}

\bigskip

\begin{lemma}
The principal eigenvalue $\lambda_0(\gamma,\mu)$  is the smallest positive solution $\lambda$ to the equation
\begin{equation}\label{lambda}
\lambda=\frac{\int_Du_{\lambda,\gamma}d\mu}{\int_Dv_{\lambda,\gamma}d\mu}.
\end{equation}
In particular
\begin{equation}\label{lambda2}
\lambda_0(\gamma,\mu)=\frac{\int_Du_{\lambda_0(\gamma,\mu),\gamma}d\mu}{\int_Dv_{\lambda_0(\gamma,\mu),\gamma}d\mu}.
\end{equation}
\end{lemma}

\begin{proof}

Let $w_{\gamma}$ denote the eigenfunction corresponding to the principal eigenvalue of $L_{\gamma,\mu}$, normalized by $\int_Dw_\gamma d\mu=1$. Then $w_\gamma$ satisfies
\begin{equation*}
\begin{aligned}
\begin{cases}
&\frac12\Delta w_\gamma-\gamma V(x)w_\gamma+\gamma V(x)=-\lambda_0(\gamma,\mu)w_\gamma;\\
&w_\gamma|_{\partial D}=0;\\
&\int_D w_\gamma d\mu=1.\\
\end{cases}
\end{aligned}
\end{equation*}

From the Feynman-Kac formula, one has
\begin{equation*}
w_\gamma(x)=E_x\int^\tau_0\gamma
V(X(t))\exp({\int^t_0(\lambda_0(\gamma,\mu)-\gamma V(X(s)))ds})dt,
\end{equation*}
and then   the  normalization  condition gives
\begin{equation}\label{int}
1=E_\mu\int^\tau_0\gamma V(X(t))\exp({\int^t_0(\lambda_0(\gamma,\mu)-\gamma V(X(s)))ds})dt.
\end{equation}
Integrating by parts gives
\begin{equation*}
\begin{aligned}
&\int^\tau_0\gamma V(X(t))\exp(\int^t_0(\lambda_0(\gamma,\mu)-\gamma V(X(s)))ds)dt=\\
&\int^\tau_0\exp(\lambda_0(\gamma,\mu)t)\frac{d}{dt}[-\exp(-\int^t_0\gamma V(X(s))ds)]dt=\\
&-\exp(\int^\tau_0(\lambda_0(\gamma,\mu)-\gamma V(X(s)))ds)+1\\
&+\lambda_0(\gamma,\mu)\int^\tau_0\exp(\int^t_0(\lambda_0(\gamma,\mu)-\gamma V(X(s)))ds)dt.\\
\end{aligned}
\end{equation*}
Substituting this in \eqref{int}, we obtain
\begin{equation}\label{eq}
\lambda_0(\gamma,\mu)=\frac{E_\mu \exp(\int^\tau_0(\lambda_0(\gamma,\mu)-\gamma V(X(s)))ds)}{
E_\mu \int^\tau_0\exp(\int^t_0(\lambda_0(\gamma,\mu)-\gamma V(X(s)))ds)dt}.
\end{equation}
By the Feynman-Kac formula again, we have that
\begin{equation}\label{us}
u_{\lambda_0(\gamma,\mu),\gamma}(x)=E_x\exp(\int^\tau_0(\lambda_0(\gamma,\mu)-\gamma V(X(s)))ds),
\end{equation}
and
\begin{equation}\label{vs} v_{\lambda_0(\gamma,\mu),\gamma}(x)=E_x\int^\tau_0\exp(\int^t_0(\lambda_0(\gamma,\mu)-\gamma V(X(s)))ds)dt.
\end{equation}
Thus, by \eqref{eq}, $\lambda_0(\gamma,\mu)$ is a positive solution to \eqref{lambda}.

Conversely, working backwards, if $\lambda$ is a solution to \eqref{lambda} (or equivalently, of \eqref{eq} with $\lambda_0(\gamma,\mu)$
replaced by $\lambda$), then it is an eigenvalue.
\end{proof}

The following lemma plays a crucial role in both the  proof of Proposition 1 and the proof of Theorem 1.

\begin{lemma}\label{sublin}

The principal eigenvalue satisfies
\begin{equation}\label{lim}
\lim_{\gamma\rightarrow\infty}\frac{\lambda_0(\gamma,\mu)}{\gamma}=0.
\end{equation}

\end{lemma}

\begin{proof}

Using \eqref{int} we have
\begin{equation*}
\frac1{\gamma\max V}\leq
E_\mu\int^\tau_0\exp(\int^t_0(\lambda_0(\gamma,\mu)-\gamma
V(X(s)))ds)dt \leq\frac1{\gamma\min V},
\end{equation*}
or equivalently from (\ref{vs})
\begin{equation*}
\frac1{\gamma\max V}\leq\int_Dv_{\lambda_0(\gamma,\mu),\gamma} d\mu\leq\frac1{\gamma\min V}.
\end{equation*}
Multiplying above by $\lambda_0(\gamma,\mu)$, (\ref{lambda2}) gives
\begin{equation}\label{uineq}
\frac{\lambda_0(\gamma,\mu)}{\gamma\max V}\leq\int_Du_{\lambda_0(\gamma,\mu),\gamma}d\mu
\leq\frac{\lambda_0(\gamma,\mu)}{\gamma\min V}.
\end{equation}

Let $p$ be an accumulation point of $\frac{\lambda_0(\gamma,\mu)}{\gamma}$ as $\gamma\to\infty$.
We note that if $p=\infty$, then for certain large $\gamma$
we would have $\int^\tau_0(\lambda_0(\gamma,\mu)-\gamma V(X(s)))ds>\lambda^D_0\tau$,
where $\lambda^D_0$ is the principal eigenvalue for the operator $-\frac12\Delta$ in $D$ with the Dirichlet boundary condition.
Using the representation in (\ref{us})
and \cite[chapter 3]{P95}, this would give $u_{\lambda_0(\gamma,\mu),\gamma}\equiv\infty$,
contradicting (\ref{uineq}).

Now assume that $0<p<\infty$. Let $\{\gamma_n\}$ be such that $p=\lim_{n\to\infty}\frac{\lambda_0(\gamma_n,\mu)}{\gamma_n}$.
As was shown in Lemma 1,  $\lambda_0(\gamma,\mu)$ is the smallest positive solution of the equation
\begin{equation}\label{lambdaintermsofE}
\lambda=\frac{E_\mu\exp(\int^\tau_0(\lambda-\gamma V(X(s)))ds)}
{E_\mu\int^\tau_0\exp(\int^t_0(\lambda-\gamma V(X(s)))ds)dt}.
\end{equation}
Both sides of \eqref{lambdaintermsofE} are continuous in $\lambda$, and since the left hand side is zero at $\lambda=0$
and the right hand side is positive at $\lambda=0$, it
follows that  for $\lambda<\lambda_0(\gamma,\mu)$ the left hand side is smaller the the right hand side.
 Fix $q\in(0,\min(p,\min V))$.
 Since $q\gamma_n\leq\lambda_0(\gamma_n)$ for sufficiently large $n$,  we  have
\begin{equation*}
\begin{aligned}
&q\gamma_n
\leq\frac{E_\mu\exp(\int^\tau_0(q\gamma_n-\gamma_n V(X(s)))ds)}
{E_\mu\int^\tau_0\exp(\int^t_0(q\gamma_n-\gamma_n V(X(s)))ds)dt}\le\\
&\frac{E_\mu\exp(-(\min V-q)\gamma_n\tau)}
{E_\mu\int^\tau_0\exp(-(\max V-q)\gamma_nt)dt}
=\frac{\gamma_n(\max V-q)
E_\mu\exp(-(\min V-q)\gamma_n\tau)}
{1-E_\mu\exp(-(\max V-q)\gamma_n\tau)}.\\
\end{aligned}
\end{equation*}

\noindent Or equivalently,
\begin{equation}\label{contradiction}
\frac q{(\max V-q)}\leq
\frac{E_\mu\exp(-(\min V-q)\gamma_n\tau)}
{1-E_\mu\exp(-(\max V-q)\gamma_n\tau)}.
\end{equation}
Letting $n\to\infty$, the right hand side of (\ref{contradiction}) goes to zero by the bounded convergence theorem, and this is a contradiction.
We have now shown that there are no accumulation points $p\in(0,\infty]$, which proves \eqref{lim}.
\end{proof}

\bigskip

We  now  prove Proposition \ref{decay}.

\it\noindent Proof of Proposition \ref{decay}.\rm\
Define
\begin{equation*}
H_\gamma(\lambda)=\frac{E_\mu\exp(\int^\tau_0(\lambda-\gamma V(X(s)))ds)}
{E_\mu\int^\tau_0\exp(\int^t_0(\lambda-\gamma V(X(s)))ds)dt}.
\end{equation*}
By Lemma 1, $\lambda_0(\gamma,\mu)$ is the smallest positive solution to the equation $H_\gamma(\lambda)=\lambda$. We now  show
 that the smallest positive solution to $H_\gamma(\lambda)=\lambda$ satisfies the upper bound in Proposition \ref{decay}.
 Let $H_\gamma^+(\lambda)$ satisfy
\begin{equation*}
H_\gamma(\lambda)\leq H_\gamma^+(\lambda),\ \forall\lambda>0,
\end{equation*}
and let $\lambda^+(\gamma)$ be the smallest positive solution to $H_\gamma^+(\lambda)=\lambda$. If we define
\begin{equation*}
G_\gamma(\lambda)=H_\gamma(\lambda)-\lambda,
\end{equation*}
then from the definition of $H_\gamma(\lambda)$, one has
$G_\gamma(0)>0$
and
\begin{equation*}
G_\gamma(\lambda^+(\gamma))=H_\gamma(\lambda^+(\gamma))-\lambda^+(\gamma)\leq H_\gamma^+(\lambda^+(\gamma))-\lambda^+(\gamma)=0.
\end{equation*}
Since $\lambda_0(\gamma,\mu)$ is the first positive zero of $G_\gamma$, it follows that $\lambda_0(\gamma,\mu)\leq\lambda^+(\gamma)$.
Thus, it suffices to show that $\lambda^+(\gamma)$  (with an appropriate choice of $H_\gamma^+$).
satisfies the upper bound in Proposition \ref{decay}. Note for use below that the above argument does not even require that $H^+_\gamma$ be continuous
or monotone, just
that there be a smallest positive root to the equation $H^+_\gamma(\lambda)=\lambda$.

To find an appropriate $H_\gamma^+$, we write
\begin{equation*}
\begin{aligned}
H_\gamma(\lambda)&=\frac{E_\mu\exp(\int^\tau_0(\lambda-\gamma V(X(s)))ds)}
{E_\mu\int^\tau_0\exp(\int^t_0(\lambda-\gamma V(X(s)))ds)dt}
\leq\frac{E_\mu \exp((\lambda-\gamma\min V)\tau)}
{E_\mu\int^\tau_0\exp((\lambda-\gamma\max V)t)dt}\\
&=\frac{(\gamma\max V-\lambda)E_\mu \exp(\gamma\min V(\frac{\lambda}{\gamma\min V}-1)\tau)}
{1-E_\mu\exp(\gamma\max V(\frac{\lambda}{\gamma\max V}-1)\tau)}\\
&\leq\gamma\max V\frac{E_\mu \exp(\gamma\min V(\frac{\lambda}{\gamma\min V}-1)\tau)}
{1-E_\mu \exp(\gamma\max V(\frac{\lambda}{\gamma\max V}-1)\tau)}.\\
\end{aligned}
\end{equation*}
By the bounded convergence theorem, $E_\mu \exp(\gamma\max V(\frac{\lambda}{\gamma\max V}-1)\tau)\le\frac12$, for $\lambda\le1$ and $\gamma$ sufficiently large.
Also, since $\mu$ is compactly supported, there exists a $c_0>0$ such that
$E_\mu \exp(\gamma\min V(\frac{\lambda}{\gamma\min V}-1)\tau)\le \exp(-c_0\gamma^\frac12)$, for
$\lambda\le1$ and $\gamma$ sufficiently large (see \cite[equation (3.3)]{P09}).
Thus, there exists a $c_1>0$ such that $H_\gamma(\lambda)\le\exp(-c_1\gamma^\frac12)$, for $\lambda\le1$ and $\gamma$ sufficiently large.
For sufficiently large $\gamma$, we now define $H^+_\gamma(\lambda)=\exp(-c_1\gamma^\frac12)$, for $\lambda\le1$, and
$H^+_\gamma(\lambda)=H_\gamma(\lambda)$, for $\lambda>1$. Then the smallest positive to the equation $H^+(\lambda)=\lambda$
is $\lambda^+(\gamma)=\exp(-c_1\gamma^\frac12)$. This gives the upper bound in the proposition.

The lower
bound is proved similarly using a function $H^-$ satisfying $H^-\le H$. At the point where \cite[equation 3.3]{P09} was used above,
one uses instead \cite[equation 3.6]{P09}. We leave the details to the reader.
\hfill $\square$

\section{Proof of Theorem \ref{th}}
We will use \eqref{lambda2} to evaluate the asymptotic behavior of $\lambda_0(\gamma,\mu)$.
The behavior of the denominator in \eqref{lambda2} is easy.
\begin{lemma}\label{vasym}
For all $\mu\in \mathcal{P}(D)$,
\begin{equation}\label{ptwise}
\lim_{\gamma\to\infty}\gamma v_{\lambda(\gamma,\mu),\gamma}=\frac1V,\ \text{boundedly pointwise in}\ D.
\end{equation}
Thus,
\begin{equation}\label{denomlim}
\lim_{\gamma\to\infty}\gamma\int_Dv_{\lambda_0(\gamma,\mu),\gamma}d\mu=\int_D\frac1Vd\mu.
\end{equation}
\end{lemma}

\begin{proof}
For notational convenience we write  $v_\gamma\equiv v_{\lambda_0(\gamma,\mu),\gamma}$.
Define $z_\gamma=\frac1{\gamma V-\lambda_0(\gamma,\mu)}$, $w_\gamma= v_\gamma-z_\gamma$ and
 $y_\gamma=\frac12\Delta z_\gamma$.  Then using \eqref{ve},  $w_\gamma$ solves the equation
\begin{equation*}
\begin{aligned}
&\frac12\Delta w_\gamma+(\lambda_0(\gamma,\mu)-\gamma  V)w_\gamma=-y_\gamma\ \text{in}\ D;\\
& w_\gamma=\frac1{\lambda_0(\gamma,\mu)-\gamma  V}\ \text{on}\ \partial D.
\end{aligned}
\end{equation*}
From the Feynman-Kac formula, one has
\begin{equation}\label{diff}
\begin{aligned}
w_\gamma(x)
&=E_x\frac1{\lambda_0(\gamma,\mu)-\gamma  V(X(\tau))}\exp(\int^\tau_0(\lambda_0(\gamma,\mu)-\gamma  V(X(t)))dt)\\
&+E_x\int^\tau_0y_\gamma(X(t))\exp(\int^t_0(\lambda_0(\gamma)-\gamma  V(X(s)))ds)dt.\\
\end{aligned}
\end{equation}
By Lemma \ref{sublin},  $\gamma y_\gamma$  and $\frac\gamma{\lambda_0(\gamma,\mu)-\gamma  V}$ are  bounded
as $\gamma\to\infty$, and  $\lambda_0(\gamma,\mu)-\gamma  V\le -\frac12(\min  V)\gamma$,
 for large $\gamma$. Thus, for some $C>0$, we have
\begin{equation*}
|\gamma w_\gamma(x)|\le CE_x\exp(-\frac12(\min V)\gamma\tau)+CE_x\int_0^\tau\exp(-\frac12(\min V)\gamma t)dt.
\end{equation*}
Thus, $\lim_{\gamma\to\infty}|\gamma w_\gamma(x)|=0$, which along with Lemma \ref{sublin}
gives \eqref{ptwise}.
\end{proof}
\medskip
We now turn to the analysis of the numerator in \eqref{lambda2}.
We will need the following key result, essentially from \cite{P09}.
\begin{lemma}\label{key}
Let $x_0 \in \partial D$ and let $n$ denote the inward unit normal to $D$ at $\partial D$. Then
\begin{equation*}
\lim_{\gamma\rightarrow\infty}\gamma^{-\frac{1}2}(\nabla u_{\lambda_0(\gamma,\mu),\gamma} \cdot n)(x_0)=-\sqrt{2V(x_0)}.
\end{equation*}
\end{lemma}
\begin{proof}
The result was proved in \cite{P09} for the case that $V=1$. From the scaling, it follows that the result continues to hold for
$V$ equal to any constant. The method of proof in \cite{P09} used localization; only the behavior near $x_0$ of the coefficients of the differential equation
solved by $u_{\lambda_0(\gamma,\gamma),\gamma}$ are relevant.
In addition, by the maximum principle,  the solutions $u_{\lambda,\gamma,\epsilon}$
to \eqref{ue} with $V(x)$ replaced by the  $V-\epsilon$
are pointwise monotone increasing in $\epsilon\in R$, and for $\gamma$ sufficiently large so that
$\lambda-\gamma V(x)$ is everywhere non-positive, they are also bounded above by 1.
Thus, $(\nabla u_{\lambda,\gamma,\epsilon}\cdot n)(x_0)$  is
monotone decreasing in $\epsilon\in R$. From these facts one deduces the lemma.
\end{proof}

 By assumption, the measure $\mu$ can be written as $\mu=\mu_{\text{cs}}+\mu_{\text{reg}}$,
where $\mu_{cs}$ is a compactly supported sub-probability measure and $\mu_{\text{reg}}$ is a sub-probability measure possessing
a density which satisfies the smoothness conditions in the statement of the theorem, and which we will also denote by $\mu_{\text{reg}}$.
Note that $\mu_{\text{reg}}(x)$ restricted to $D^\epsilon$ coincides with the density $\mu(x)$ appearing in the statement of the theorem.
 We write \eqref{lambda2} as
\begin{equation}\label{mureg}
\lambda_0(\gamma,\mu)=
\frac{\int_Du_{\lambda_0(\gamma,\mu),\gamma}\mu_{\text{reg}}dx}
{\int_Dv_{\lambda_0(\gamma,\mu),\gamma}d\mu}+
\frac{\int_Du_{\lambda_0(\gamma,\mu),\gamma}d\mu_{\text{cs}}}
{\int_Dv_{\lambda_0(\gamma,\mu),\gamma}d\mu}.
\end{equation}
Using \eqref{us} and Lemma \ref{sublin} for the first inequality below, and \cite[equation 3.3]{P09} and the fact that $\mu_{\text{cs}}$ is compactly supported
for the second one, one has for some $c>0$ and large $\gamma$,
\begin{equation}\label{hittingtime}
u_{\lambda_0(\gamma,\mu),\gamma}(x)\le E_x\exp(-\frac12(\min V)\gamma\tau_D)\le \exp(-c\gamma^\frac12),\ \text{for all}\
x\in\text{supp}(\mu_{\text{cs}}).
\end{equation}
From \eqref{hittingtime} and Lemma \ref{vasym}, there exists a $C>0$ such that for large $\gamma$,
\begin{equation}\label{negl}
\frac{\int_Du_{\lambda_0(\gamma,\mu),\gamma}d\mu_{\text{cs}}}
{\int_Dv_{\lambda_0(\gamma,\mu),\gamma}d\mu}
\le \exp(-C\gamma^\frac12).
\end{equation}

In the statement of the theorem, note that \eqref{1} is in fact the same as \eqref{even} with $k=0$.
(We simply separated this case out for the sake of exposition.)
Thus to prove the theorem, we must show that \eqref{even} holds for even $k\ge0$ and that \eqref{odd} holds for odd $k\ge1$.
In light of  \eqref{mureg}, \eqref{negl} and Lemma \ref{vasym}, the theorem will be proved if we show
that
\begin{equation}\label{finalodd}
\lim_{\gamma\to\infty}\gamma^{\frac{k+1}2}
\int_Du_{\lambda_0(\gamma,\mu),\gamma}\mu_{\text{reg}}dx=
\frac{\int_{\partial D}
V^{-\frac{k+1}2}\nabla(\Delta^{\frac{k-1}2}\mu)\cdot nd\sigma}{2^{\frac{k+1}2}}, \ \text{for odd}\ k\ge1,
\end{equation}
and that
\begin{equation}\label{finaleven}
\lim_{\gamma\to\infty}\gamma^{\frac{k+1}2}\int_Du_{\lambda_0(\gamma,\mu),\gamma}\mu_{\text{reg}}dx=
\frac{\int_{\partial D}
V^{-\frac{k+1}2}\Delta^{\frac k2}\mu d\sigma}{2^{\frac{ k+1}2}}, \ \text{for even}\ k\ge0.
\end{equation}

Although we could give   a steam-lined proof that works simultaneously for all  even $k$ and another
one that works for all odd $k$, we prefer the following route, in the interest of clarity of exposition.
We will first show \eqref{finalodd} for $k=1$ and \eqref{finaleven} for $k=0$.
Then we will show how to iterate
the method for $k=1$ to obtain \eqref{finalodd} for $k=3$ and  will note how to continue for general odd $k$. Then we will show  how to iterate
the method for $k=0$ to obtain \eqref{finaleven} for $k=2$ and will note how to continue for general even $k$.

We begin with $k=1$, which is easier than $k=0$. Recalling that $n$ denotes the unit inward normal,
and using \eqref{ue} and the fact that $\mu_{\text{reg}}$ vanishes on $\partial D$, integration by parts gives
\begin{equation}\label{k=1}
\begin{aligned}
&\gamma\int_Du_{\lambda_0(\gamma,\mu),\gamma}\mu_{\text{reg}}dx=\int_D(\frac12\Delta u_{\lambda_0(\gamma,\mu),\gamma})
\frac{\gamma \mu_{\text{reg}}}{\gamma V-\lambda_0(\gamma,\mu)}dx=\\
&\frac12\int_Du_{\lambda_0(\gamma,\mu),\gamma}
(\Delta\frac{\gamma\mu_{\text{reg}}}{\gamma V-\lambda_0(\gamma,\mu)})dx+
\frac12\int_{\partial D}\nabla(\frac{\gamma\mu_{\text{reg}}}{\gamma V-\lambda_0(\gamma,\mu)})\cdot nd\sigma.
\end{aligned}
\end{equation}
By Lemma \ref{sublin}, as $\gamma\to\infty$, $\Delta\frac{\gamma\mu_{\text{reg}}}{\gamma V-\lambda_0(\gamma,\mu)}$
converges boundedly pointwise to $\Delta\frac{\mu_{\text{reg}}}V$.
By \eqref{hittingtime}, $u_{\lambda_0(\gamma,\mu),\gamma}$ converges boundedly pointwise to 0 in $D$.
Also, $\nabla(\frac{\gamma\mu_{\text{reg}}}{\gamma V-\lambda_0(\gamma,\mu)})\cdot n$ converges boundedly pointwise
to $\nabla\frac{\mu_{\text{reg}}}V$ on $\partial D$.
Since $\mu_{\text{reg}}$ vanishes on $\partial D$, one has $\nabla\frac{\mu_{\text{reg}}}V=V^{-1}\nabla
\mu_{\text{reg}}$ on $\partial D$. Using these facts and
 letting $\gamma\to\infty$ in \eqref{k=1} gives \eqref{finalodd} for $k=1$.

We now turn to the case $k=0$. For large $\gamma$, let $w_\gamma$ solve the equation
\begin{equation}\label{aux}
\begin{aligned}
&\Delta \frac{w_\gamma}{\gamma V-\lambda_0(\gamma,\mu)}=0\ \text{in}\ D;\\
&w_\gamma=\mu_{\text{reg}}\ \text{on}\ \partial D.
\end{aligned}
\end{equation}
We will show below that
\begin{equation}\label{diff}
\lim_{\gamma\to\infty}\gamma^\frac12\int_Du_{\lambda_0(\gamma,\mu),\gamma}(\mu_{\text{reg}}-w_\gamma)dx=0.
\end{equation}
Thus, it is enough to show \eqref{finaleven} for $k=0$ with $\mu_{\text{reg}}$
replaced by $w_\gamma$.
Using \eqref{ue} and \eqref{aux}, and integrating by parts, we have
\begin{equation}\label{k=0}
\begin{aligned}
&\gamma^{\frac12}\int_Du_{\lambda_0(\gamma,\mu),\gamma}w_\gamma dx=\gamma^{-\frac12}\int_D
(\frac12\Delta u_{\lambda_0(\gamma,\mu),\gamma})\frac{\gamma w_\gamma}{\gamma V-\lambda_0(\gamma,\mu)}dx=\\
&-\frac{\gamma^{-\frac12}}2\int_{\partial D}\frac{\gamma \mu_{\text{reg}}}{\gamma V-\lambda_0(\gamma,\mu)}
(\nabla u_{\lambda_0(\gamma,\mu),\gamma}\cdot n) d\sigma,
\end{aligned}
\end{equation}
where we have used the fact that
$$
\int_{\partial D}\nabla( \frac{\gamma w_\gamma}{\gamma V-\lambda_0(\gamma,\mu)})\cdot nd\sigma=
\int_D\Delta \frac{\gamma w_\gamma}{\gamma V-\lambda_0(\gamma,\mu)}dx=0.
$$
Letting $\gamma\to\infty$ in \eqref{k=0} and using Lemma \ref{key} and
Lemma \ref{sublin}, we obtain
\begin{equation}\label{k=00}
\lim_{\gamma\to\infty}\gamma^{\frac12}\int_Du_{\lambda_0(\gamma,\mu),\gamma}w_\gamma dx=
\frac{\int_{\partial D}
\frac{\mu_{\text{reg}}}{\sqrt V}d\sigma}{2^{\frac12}},
\end{equation}
which is \eqref{finaleven} for $k=0$ with $\mu_{\text{reg}}$ replaced by $w_\gamma$.

To complete the proof of the case $k=0$, we now prove
\eqref{diff}.
For $\epsilon>0$, we have
\begin{equation}\label{diffeps}
|\gamma^\frac12\int_{D^\epsilon}u_{\lambda_0(\gamma,\mu),\gamma}(\mu_{\text{reg}}-w_\gamma)dx|\le\sup_{x\in D^\epsilon}
|\mu_{\text{reg}}(x)-w_\gamma(x)|
(\gamma^\frac12\int_Du_{\lambda_0(\gamma,\mu),\gamma}dx).
\end{equation}
Note that by Lemma \ref{sublin},  $w\equiv\lim_{\gamma\to\infty}w_\gamma$ solves $\Delta_Vw=0$ in $D$ and $w=\mu_{\text{reg}}$
on $\partial D$.
From standard results, it then follows that
\begin{equation}\label{diffbound}
\lim_{\epsilon\to0}\sup_{x\in D^\epsilon}
|\mu_{\text{reg}}(x)-w_\gamma(x)|=0.
\end{equation}
In the  case that, say, $\mu_{\text{reg}}\equiv1$ on $\partial D$, it follows from the maximum principle that
$w_\gamma$ is strictly positive in $\bar D$, uniformly over large $\gamma$.
By the maximum principle and Lemma \ref{sublin}, $u_{\lambda_0(\gamma,\mu),\gamma}$ is decreasing in $\gamma$, for large $\gamma$.
Using these facts with \eqref{k=00},
 it follows that for \it all\rm\ choices of $\mu$, one has that  $\gamma^\frac12\int_Du_{\lambda_0(\gamma,\mu),\gamma} dx$ is bounded as $\gamma\to\infty$.
Using this with \eqref{diffbound}, it follows from \eqref{diffeps} that
\begin{equation}\label{finaldiff}
\lim_{\epsilon\to0}\limsup_{\gamma\to\infty}
|\gamma^\frac12\int_{D^\epsilon}u_{\lambda_0(\gamma,\mu),\gamma}(\mu_{\text{reg}}-w_\gamma)dx|=0.
\end{equation}
By \eqref{hittingtime} and the uniform boundedness in $\gamma$
of $w_\gamma$, it follows that
\begin{equation}\label{inside}
\lim_{\gamma\to\infty}\gamma^\frac12\int_{D-D^\epsilon}u_{\lambda_0(\gamma,\mu),\gamma}(\mu_{\text{reg}}-w_\gamma)dx=0.
\end{equation}
Now \eqref{diff} follows from \eqref{finaldiff} and \eqref{inside}.

We now consider the cases $k=2$ and $k=3$, beginning with $k=3$.
In the case $k=3$, $\mu$ and all its derivatives up to order 2 vanish on $\partial D$;
in particular, the last term on the right hand side of \eqref{k=1} is 0.
Thus, using \eqref{ue} again, integrating by parts and using the fact that the second order derivatives of $\mu$ vanish
on $\partial D$, we have from \eqref{k=1}
\begin{equation}\label{k=3}
\begin{aligned}
&\gamma^2\int_Du_{\lambda_0(\gamma,\mu),\gamma}\mu_{\text{reg}}dx=
\frac\gamma2\int_Du_{\lambda_0(\gamma,\mu),\gamma}(\Delta\frac{\gamma\mu_{\text{reg}}}{\gamma V-\lambda_0(\gamma,\mu)})dx=\\
&\frac12\int_D(\frac12\Delta u_{\lambda_0(\gamma,\mu),\gamma})\frac\gamma{\gamma V-\lambda_0(\gamma,\mu)}
(\Delta\frac{\gamma\mu_{\text{reg}}}{\gamma V-\lambda_0(\gamma,\mu)})dx=\\
&\frac1{2^2}\int_Du_{\lambda_0(\gamma,\mu),\gamma}(\Delta\frac\gamma{\gamma V-\lambda_0(\gamma,\mu)}
\Delta\frac{\gamma\mu_{\text{reg}}}{\gamma V-\lambda_0(\gamma,\mu)})dx+\\
&\frac1{2^2}\int_{\partial D}
\nabla(\frac\gamma{\gamma V-\lambda_0(\gamma,\mu)}
\Delta\frac{\gamma\mu_{\text{reg}}}{\gamma V-\lambda_0(\gamma,\mu)})\cdot nd\sigma.
\end{aligned}
\end{equation}
By Lemma \ref{sublin}, as $\gamma\to\infty$, $\Delta\frac\gamma{\gamma V-\lambda_0(\gamma,\mu)}
\Delta\frac{\gamma\mu_{\text{reg}}}{\gamma V-\lambda_0(\gamma,\mu)}$ converges boundedly pointwise to $\Delta\frac1V\Delta\frac{\mu_{\text{reg}}}V$, and
by \eqref{hittingtime}, $u_{\lambda_0(\gamma,\mu),\gamma}$ converges boundedly pointwise to 0.
Also, $\nabla(\frac\gamma{\gamma V-\lambda_0(\gamma,\mu)}
\Delta\frac{\gamma\mu_{\text{reg}}}{\gamma V-\lambda_0(\gamma,\mu)})\cdot n$ converges boundedly pointwise on $\partial D$
to $\nabla(\frac1V
\Delta\frac{\mu_{\text{reg}}}V)\cdot n$.
Since $\mu$ and all of its derivatives up to order 2 vanish on $\partial D$,
one has $\nabla(\frac1V
\Delta\frac{\mu_{\text{reg}}}V)=V^{-2}\nabla(\Delta \mu_{\text{reg}})$ on $\partial D$.
Thus, letting $\gamma\to\infty$ in \eqref{k=3} gives
\eqref{finalodd} for $k=3$.
Note that in \eqref{k=3} we needed $\mu_{\text{reg}}$ and $V$ to be 4 times differentiable.
 When $k=5$, the boundary term on the right hand side of \eqref{k=3} vanishes, and one
again uses \eqref{ue} to replace $u_{\lambda_0(\gamma,\mu),\gamma}$ in the first term on the right hand
side of \eqref{k=3} by $(\frac12\Delta u_{\lambda_0(\gamma,\mu),\gamma})\frac1{\gamma V-\lambda_0(\gamma,\mu)}$. This time the calculations requires that $\mu_{\text{reg}}$ and $V$  be 6 times differentiable.
It should be clear how to continue for all odd $k$.

We now consider the case $k=2$. Since $\mu$ and its first order derivatives vanish on $\partial D$, we have
from \eqref{k=1}
\begin{equation}\label{3.9}
\gamma^\frac32\int_Du_{\lambda_0(\gamma,\mu),\gamma}\mu_{\text{reg}}dx=\frac{\gamma^\frac12}2\int_Du_{\lambda_0(\gamma,\mu),\gamma}
(\Delta\frac{\gamma\mu_{\text{reg}}}{\gamma V-\lambda_0(\gamma,\mu)})dx.
\end{equation}
As with the case $k=0$, we define an auxiliary function $w_\gamma$, which satisfies this time the equation
\begin{equation}\label{auxagain}
\begin{aligned}
&\Delta \frac{w_\gamma}{\gamma V-\lambda_0(\gamma,\mu)}=0\ \text{in}\ D;\\
&w_\gamma=\Delta\frac{\gamma\mu_{\text{reg}}}{\gamma V-\lambda_0(\gamma,\mu)}\ \text{on}\ \partial D.
\end{aligned}
\end{equation}
The same argument used to show \eqref{diff} shows that
\begin{equation}\label{diffagain}
\lim_{\gamma\to\infty}\gamma^\frac12\int_Du_{\lambda_0(\gamma,\mu),\gamma}
(w_\gamma-\Delta\frac{\gamma\mu_{\text{reg}}}{\gamma V-\lambda_0(\gamma,\mu)})dx=0.
\end{equation}
From \eqref{3.9} and \eqref{diffagain} we have
\begin{equation}\label{reduced}
\lim_{\gamma\to\infty}\gamma^\frac32\int_Du_{\lambda_0(\gamma,\mu),\gamma}\mu_{\text{reg}}dx=
\lim_{\gamma\to\infty}\frac{\gamma^\frac12}2\int_Du_{\lambda_0(\gamma,\mu),\gamma}w_\gamma dx.
\end{equation}
Using \eqref{ue} and \eqref{auxagain}, and integrating by parts, we have
\begin{equation}\label{final}
\begin{aligned}
&\frac{\gamma^\frac12}2\int_Du_{\lambda_0(\gamma,\mu),\gamma}w_\gamma dx=
\frac{\gamma^{-\frac12}}2\int_D(\frac12\Delta u_{\lambda_0(\gamma,\mu),\gamma})\frac{\gamma w_\gamma}{\gamma V-\lambda_0(\gamma,\mu)}dx=\\
&-\frac{\gamma^{-\frac12}}4\int_{\partial D}(\nabla u_{\lambda_0(\gamma,\mu),\gamma}\cdot n)\frac{\gamma w_\gamma}{\gamma V-\lambda_0(\gamma,\mu)}d\sigma=\\
&-\frac{\gamma^{-\frac12}}4\int_{\partial D}(\nabla u_{\lambda_0(\gamma,\mu),\gamma}\cdot n)
\frac{\gamma }{\gamma V-\lambda_0(\gamma,\mu)}\Delta\frac{\gamma\mu_{\text{reg}}}{\gamma V-\lambda_0(\gamma,\mu)}d\sigma,
\end{aligned}
\end{equation}
where we have used the fact that
$$
\int_{\partial D}\nabla( \frac{\gamma w_\gamma}{\gamma V-\lambda_0(\gamma,\mu)})\cdot nd\sigma=
\int_D\Delta \frac{\gamma w_\gamma}{\gamma V-\lambda_0(\gamma,\mu)}dx=0.
$$
Note that since $\mu$ and its first derivatives vanish on $\partial D$, one has
$\Delta\frac{\mu_{\text{reg}}}V=V^{-1}\Delta\mu_{\text{reg}}$ on $\partial D$. Using this and
letting $\gamma\to\infty$, it follows from \eqref{reduced}, \eqref{final} and Lemma \ref{key} that
\begin{equation}
\lim_{\gamma\to\infty}\gamma^\frac32\int_Du_{\lambda_0(\gamma,\mu),\gamma}\mu_{\text{reg}}dx=
\frac{\int_{\partial D}V^{-\frac32}
\Delta\mu_{\text{reg}}d\sigma}{2^\frac32},
\end{equation}
which is \eqref{finaleven} for $k=2$.
It should be clear how to continue in the same vein for larger even $k$.

We now  explain the smoothness requirement in the case of $k=2$, $k=0$ and then for higher order $k$.
First consider $k=2$.
In order to apply the divergence theorem in \eqref{final} we needed for $w_\gamma$ to be in $C^2(D)\cap C^1(\bar D)$.
For this, we claim that it suffices to have   $\mu_{\text{reg}}\in C^3(\bar D)$ and $V\in C^{3,\alpha}(\bar D)$.
To see this, recall that the standard theory \cite{GT} guarantees that if $L$ is a second-order elliptic operator, then
the equation  $Lu=f$ in $D$ and $u=\phi$ on $\partial D$ has a
solution $u\in C^{2,\alpha}(D)\cap C(\bar D)$ if $f$ and  the coefficients of $L$ are in $C^\alpha(\bar D)$, $\phi$
is continuous and $\partial D$
is a $C^{2,\alpha}$-boundary.  Thus,  by the above smoothness assumptions on $\mu_{\text{reg}}$ and $w_\gamma$,
it follows from \eqref{auxagain}
that
$w_\gamma\in C^{2,\alpha}(D)\cap C(\bar D)$.
Now formally differentiate \eqref{auxagain} with respect to $x_j$, and  formally, let
  $z_\gamma=\frac{\partial w_\gamma}{\partial x_j}$.
  Using the above  smoothness of $w_\gamma$,   and again
  using the above smoothness assumptions on $\mu_{\text{reg}}$ and $V$, one has  formally that $z_\gamma$ satisfies an equation
  of the form   $L_1z_\gamma=f$ in $D$ and
  $z_\gamma=\phi$ on $\partial D$, where $f$ and the coefficients of the operator $L_1$
  belong to $C^\alpha(\bar D)$, and $\phi$ is continuous.
Thus, by the  general theory, the above equation has a solution $z_\gamma\in C^{2,\alpha}(D)\cap C(\bar D)$.
One then shows  that $z_\gamma$ is in fact $\frac{\partial w_\gamma}{\partial x_j}$, which
establishes that $w_\gamma$ is in $C^1(\bar D)$.

For $k=0$, the auxiliary function $w_\gamma$ solves \eqref{aux}. By the line of reasoning in the above paragraph, one
needs $\mu_{\text{reg}}\in C^1(\bar D)$ and $V\in C^{2,\alpha}(\bar D)$.
For higher order even $k$, the auxiliary function $w_\gamma$
that one constructs solves the equation
\begin{equation}
\begin{aligned}
&\Delta \frac{w_\gamma}{\gamma V-\lambda_0(\gamma,\mu)}=0\ \text{in}\ D;\\
&w_\gamma=\Delta_{\gamma,V}^{\frac k2}\mu_{\text{reg}}\ \text{on}\ \partial D,
\end{aligned}
\end{equation}
where the operator $\Delta_{\gamma,V}$ is defined by $\Delta_{\gamma,V}g=\Delta\frac{\gamma g}{\gamma V-\lambda_0(\gamma,\mu)}$.
By the reasoning of the previous paragraph,
one needs
$\mu_{\text{reg}}\in C^{k+1}(\bar D)$ and  $V\in C^{k+1,\alpha}(\bar D)$.

\section{An Open Problem in a Degenerate Case}

Consider the case that $V$ is positive in $D$ but vanishes on $\partial D$.
As was noted in the penultimate paragraph before  Theorem \ref{th},
if $V$ decays to 0 at the boundary  at an appropriate rate, it should increase the tendency of the process
to leave the region, and thus raise the value of $\lambda_0(\gamma,\mu)$.
Indeed, in order for the process to exit the region, when the process is very near the boundary it needs to refrain from jumping.
And as was noted in the second paragraph of the remark after Theorem \ref{th},
assuming that $\mu\not\equiv 0$ on $\partial D$,  if $V$ is of the form $\epsilon+(1-\epsilon|D|)\hat V$, where $\hat V$ is a smooth function which is strictly positive in $D$ and vanishes on $\partial D$,
then as $\epsilon\to0$, the right hand side of \eqref{1} converges to $\infty$.
These facts suggest that in the case that $V$ is smooth and vanishes on $\partial D$, and the density $\mu$ does not vanish
identically on $\partial D$, then $\lambda_0(\gamma,\mu)$ should grow on an order larger  than $\gamma^\frac12$ as $\gamma\to\infty$.
On the other hand, if $V$ is compactly supported in $D-D_\epsilon$,
then by the reasoning in the penultimate paragraph before Theorem \ref{th}, one has
 $\lambda_0(\gamma,\mu)\le\lambda_0^{D_\epsilon}$, where
 $\lambda_0^{D_\epsilon}$ is the principal eigenvalue for $-\frac12\Delta$ in $D_\epsilon$ with
 the Dirichlet boundary condition.
 Thus, it also seems possible that if $V$ decays to 0 at the boundary
 sufficiently fast, then in fact $\lambda_0(\gamma,\mu)$ should be of smaller  order than $\gamma^\frac12$ as $\gamma\to\infty$.

To determine what happens, it should suffice to look at the simple one-dimensional case with
$D=(0,1)$. We consider $V$ with a first-order 0 at the boundary. To make things simple, we choose $V$ symmetric: $V(x)=6x(1-x)$
(we continue with the normalization $\int_DVdx=1$). We take $\mu$ to be Lebesgue measure.
Thus, we have
\begin{equation}\label{special}
\begin{aligned}
&L_{\gamma,\mu}u(x)=-\frac12u''(x)+6\gamma x(1-x)\left(u(x)-\int_0^1u(y)dy\right)\ \text{on}\ (0,1);\\
&u(0)=u(1)=0.
\end{aligned}
\end{equation}
Unfortunately, we are only able to conclude that there exist $c_1,c_2>0$ such that
\begin{equation}\label{weakest}
c_1\gamma^\frac13\le\lambda_0(\gamma,\mu)\le c_2\gamma^\frac23.
\end{equation}

We obtain \eqref{weakest} as follows.
By the criticality theory of second order elliptic operators \cite{P95}, which can be applied to $L_{\gamma,\mu}$ as in \eqref{special},
$\lambda_0(\gamma,\mu)$ can be characterized as the supremum over those $\lambda$ for which there exists a
function $u>0$ on $D$ satisfying $L_{\gamma,\mu}u-\lambda u\ge0$ in $D$.
(It is enough to work with $C^1$ functions that are piecewise $C^2$.)
One can check that if one defines $u(x)=x-\gamma^\frac13x^2$, for $0\le x\le \frac12\gamma^{-\frac13}$, $u(x)=u(\frac12\gamma^{-\frac13})$,
for $\frac12\gamma^{-\frac13}\le x\le \frac12$, and then extends $u$ to $(0,1)$
by making it symmetric with respect to $x=\frac12$, then for sufficiently small $\epsilon>0$, one has
$L_{\gamma,\mu}u-\epsilon\gamma^\frac13 u\ge0$ in $D$. This gives the lower bound in \eqref{weakest}.

Another way to characterize $\lambda_0(\gamma,\mu)$ is that it is the largest $\lambda$ such that
the generalized maximum principle holds for $L_{\gamma,\mu}-\lambda$. That is, the largest $\lambda$ such
that whenever one has  $L_{\gamma,\mu}v-\lambda v\le 0$ in $D$ and $v(0)=v(1)=0$, then necessarily one has
$v\le0$ in $D$. Choosing $u$ as above, one can show that if
 $\epsilon>0$
is sufficiently small, then one has $L_{\gamma,\mu}u-\epsilon\gamma^\frac23 u\le 0$ in $D$. Since $u\ge0$ in $D$, the generalized
maximum principle does not hold and consequently $\lambda_0(\gamma,\mu)\le\epsilon\gamma^\frac23$, giving the upper
bound in \eqref{weakest}. We have experimented with all sorts of much more complicated functions, but have not been
able to improve the above bounds.

The upper bound in  \eqref{weakest} can be understood probabilistically by the following heuristic argument, which
may be able to be made rigorous. If a Brownian motion is at $x$, then the probability that it
will reach 0 by time $s$ is no more than $\exp(-c\frac{x^2}s)$, for some $c>0$. When the process $X(\cdot)$ is
 at $\gamma^{-l}$, the local jump rate is on the order $\gamma^{1-l}$ and thus the expected time
 to jump is on the order $\gamma^{l-1}$. Letting $s=\gamma^{l-1}$ and $x=\gamma^{-l}$, with  $l<\frac13$,
 it follows that for large $\gamma$, any time the   $X(\cdot)$ process finds itself in $[\gamma^{-l},1-\gamma^{-l}]$,
  the probability that the process will hit  0
before jumping is overwhelmingly small. On a fixed time interval $t$, one expects no more than $c\gamma t$ jumps,
for some $c>0$. The probability that all of these jumps will send the process to  $[\gamma^{-l},1-\gamma^{-l}]$ is
at least $(1-2\gamma^{-l})^{c\gamma t}$. So the probability of not exiting by time $t$ is at least on the order
$(1-2\gamma^{-l})^{c\gamma t}$, which is at least $\exp(-c_1\gamma^{1-l}t)$ for some $c_1>0$.
By \eqref{largetime}, we conclude that $\lambda_0(\gamma,\mu)$ grows no faster than $\gamma^{1-l}$ for any $l<\frac13$.

\end{document}